\documentclass[11pt, reqno]{amsart}
\usepackage{amsmath, amsfonts, amsthm, amssymb, multicol, mathtools, dsfont, mathrsfs}
\usepackage{graphicx}
\usepackage{float, hyperref}
\usepackage{scalerel,stackengine, subcaption}
\usepackage[usenames,dvipsnames,x11names]{xcolor}
\usepackage{enumitem}
\usepackage{pgfplots}
\usepgfplotslibrary{fillbetween}
\pgfplotsset{width=10cm,compat=1.9}
\usepackage[square,sort,comma,numbers]{natbib}
\allowdisplaybreaks

\makeatletter
\g@addto@macro\bfseries{\boldmath}
\makeatother

\makeatletter
\def\@setauthors{%
  \begingroup
  \def\thanks{\protect\thanks@warning}%
  \trivlist
  \centering\footnotesize \@topsep30\p@\relax
  \advance\@topsep by -\baselineskip
  \item\relax
  \author@andify\authors
  \def\\{\protect\linebreak}

  \normalsize\lowercase{\authors}%
  
	\ifx\@empty\contribs
  \else
    ,\penalty-3 \space \@setcontribs
    \@closetoccontribs
  \fi
  \endtrivlist
  \endgroup
}
\def\@settitle{\begin{center}
\LARGE\lowercase{\@title}
  \end{center}%
}
\makeatother
\newcommand{\authoremail}[1]{\email{\href{mailto:#1}{\color{lightblue}{#1}}}}
\newcommand{\authoraddress}[1]{\address{\normalfont{#1}}}

\numberwithin{equation}{section}

\newtheorem{thm}{Theorem}[section]
\newtheorem{lma}[thm]{Lemma}
\newtheorem{cor}[thm]{Corollary}

\newtheorem{prop}[thm]{Proposition}

\renewcommand{\epsilon}{\varepsilon}

\newcommand{\rd}{\mathbb{R}^d}

\renewcommand{\geq}{\geqslant}
\renewcommand{\leq}{\leqslant}

\newcommand{\lbd}{\underline{\dim}_{\textup{B}}}
\newcommand{\hd}{\dim_{\textup{H}}}
\newcommand{\frd}{\dim_{\textup{Fr}}}

\newcommand{\fs}{\dim^\theta_{\mathrm{F}}}
\newcommand{\fd}{\dim_{\mathrm{F}}}
\newcommand{\sd}{\dim_{\mathrm{S}}}

\newcommand{\R}{\mathbb{R}}
\newcommand{\E}{\mathbb{E}}

\newcommand{\C}{\mathbb{C}}

\newcommand{\spt}{\text{spt}\,}

\newcommand{\J}{\mathcal{J}}

\DeclareMathOperator*{\argmax}{arg\,max}

\makeatletter
\DeclareRobustCommand\widecheck[1]{{\mathpalette\@widecheck{#1}}}
\def\@widecheck#1#2{%
    \setbox\z@\hbox{\m@th$#1#2$}%
    \setbox\tw@\hbox{\m@th$#1%
       \widehat{%
          \vrule\@width\z@\@height\ht\z@
          \vrule\@height\z@\@width\wd\z@}$}%
    \dp\tw@-\ht\z@
    \@tempdima\ht\z@ \advance\@tempdima2\ht\tw@ \divide\@tempdima\thr@@
    \setbox\tw@\hbox{%
       \raise\@tempdima\hbox{\scalebox{1}[-1]{\lower\@tempdima\box
\tw@}}}%
    {\ooalign{\box\tw@ \cr \box\z@}}}
\makeatother

\stackMath
\newcommand\reallywidehat[1]{%
\savestack{\tmpbox}{\stretchto{%
  \scaleto{%
    \scalerel*[\widthof{\ensuremath{#1}}]{\kern.1pt\mathchar"0362\kern.1pt}%
    {\rule{0ex}{\textheight}}
  }{\textheight}%
}{2.4ex}}%
\stackon[-6.9pt]{#1}{\tmpbox}%
}
\definecolor{lightblue}{HTML}{2B77A4}
\colorlet{plotblue}{LightSkyBlue3!80}
\definecolor{darkred}{HTML}{9E0D0D}
\definecolor{purp}{HTML}{d603a9}
\definecolor{dartmouthgreen}{HTML}{00A64F}
\definecolor{Junglegreen}{HTML}{00A99A}
\definecolor{yellowcolour}{HTML}{f07c02}
\hypersetup{
	colorlinks=true,
	linkcolor=darkred,
	urlcolor=darkred,
	citecolor=lightblue
}
\newcommand\numberthis{\addtocounter{equation}{1}\tag{\theequation}}

\title{Fourier  restriction  estimates based on $L^q$-dimensions:\\ beyond Stein--Tomas}

\usepackage{fancyhdr}
\author{Marc Carnovale}
\authoraddress{Marc Carnovale}
\authoremail{marc.carnovale@gmail.com}

\author{Jonathan M. Fraser}
\authoraddress{Jonathan M. Fraser, University of St Andrews, Scotland}
\authoremail{jmf32@st-andrews.ac.uk}
\thanks{JMF was  financially supported by a  \emph{Leverhulme Trust Research Project Grant} (RPG-2023-281)  and an \emph{EPSRC Open Fellowship} (EP/Z533440/1).}

\author{Ana E. de Orellana}
\authoraddress{A. E. de Orellana, University of St Andrews, Scotland}
\authoremail{aedo1@st-andrews.ac.uk}
\thanks{AEdO was financially supported by the University of St Andrews.}

\date{}

\begin{document}
\thispagestyle{empty}

\begin{abstract}
  The well-known Stein--Tomas restriction theorem gives the sharp  range of $p$ for which $L^p\to L^2$ restriction estimates hold for the surface measure on the sphere. This was generalised to arbitrary measures satisfying certain Fourier decay and Frostman conditions by Mockenhaupt, Mitsis, and Bak--Seeger, with the most general version now a fundamental result in harmonic analysis. The Frostman condition essentially asks for uniform control on the measure of small balls and is the endpoint  of a continuum of more nuanced conditions which describe the local fluctuations of the measure. This analysis  gives rise to the  $L^q$-dimensions of a measure and these are a central concept in fractal geometry and a crucial tool in multifractal analysis and the theory of large deviations. In this paper we prove a new Fourier restriction theorem which uses the $L^q$-dimensions instead of the Frostman condition, thus providing a continuum of estimates which recover Stein--Tomas at the endpoint. Our proof gives the endpoint estimate for all values of $q\in(1,\infty]$ via Stein's complex interpolation. In particular, in the case $q=\infty$ this partially resolves    a  question raised by Bak and Seeger. We explore when our theorem improves on Stein--Tomas, that is, when the range is not optimised at $q=\infty$, and show that this is the case quite generally, including for certain  Mandelbrot cascade measures and measures with multifractal behaviour. On the way to proving our main theorem we obtain a novel description of the $L^q$-dimensions based on certain convolution norms, which may be of interest in its own right.  
\end{abstract}
\maketitle
\tableofcontents

\section{Introduction}

\subsection{The restriction problem} \label{sec:restrictionIntro}

An important  question in harmonic  analysis---and the fundamental question in Fourier restriction theory---asks whether we can meaningfully restrict the Fourier transform of a function to the support of a singular measure $\mu$ on $\rd$.  More precisely, given a non-zero, finite, compactly supported, Borel measure $\mu$ on $\rd$, for which $p,r\in[1,\infty]$ does the estimate 
\begin{equation}\label{eq:restriction}
    \|\widehat{f}\|_{L^{r'}(\mu)} \lesssim \|f\|_{L^{p'}(\rd)},
\end{equation}
hold for all $f\in L^{p'}(\rd)$? Here we use the notation $A\lesssim B$ if there exists a uniform constant $c\geq1$ such that $A\leq cB$, and $A\approx B$ if $A\lesssim B$ and $B\lesssim A$. If the constant $c$ depends on some parameter $\lambda$ we wish to emphasise, it will be denoted with a subscript as $\lesssim_{\lambda}$, or $\approx_{\lambda}$. Also, for   $p,r\in[1,\infty]$, we write $p'$ and $r'$ to refer to their conjugate exponents, i.e. they will satisfy $\frac{1}{p}+ \frac{1}{p'} = \frac{1}{r} + \frac{1}{r'} = 1$. 

By duality of $L^p$ spaces, \eqref{eq:restriction} is equivalent to
\begin{equation*}
    \|\widehat{f\mu}\|_{L^{p}(\rd)}\lesssim \|f\|_{L^r(\mu)}
\end{equation*}
with \eqref{eq:restriction} referred to as an $L^{p'} \to L^{r'}$ restriction estimate and the dual form as  an $L^{r} \to L^{p}$ extension estimate. What is more, if $p = 2$, both restriction and extension estimates are equivalent to
\begin{equation}\label{eq:extension}
    \|\widehat{\mu}*f\|_{L^p(\rd)}\lesssim \|f\|_{L^{p'}(\rd)},
\end{equation}
which we will refer to as an $L^{p'}\to L^p$ $TT^*$ estimate, since $\widehat{\mu}*f$ is the $TT^*$ operator of the restriction operator $T$.

We suppose from now on that $\mu$ is singular.  If  $p'=1$, $\widehat{f}$ is continuous and estimates of the form \eqref{eq:restriction} will be possible for all $r'$. However, since the Fourier transform is an $L^2(\rd)$ isometry, for $p'= 2$, $\widehat{f}\in L^2(\rd)$ and there is no sensible way to restrict it to a set of Lebesgue measure zero. Thus, the question is only  interesting for the range $1< p'<2$ (equivalently, $p>2$).


\subsection{Stein--Tomas and Mockenhaupt--Mitsis--Bak--Seeger restriction}

When the support of $\mu$ is a smooth manifold, Stein observed that non-trivial estimates were often possible when the manifold was curved, however \eqref{eq:restriction} could not hold for $p'<\infty$ when the manifold was flat. This curvature is captured by Fourier decay of $\mu$, and is one of the principal hypotheses in the Stein--Tomas restriction theorem. Although the name Stein--Tomas refers to work by Stein (see e.g. \cite{Ste93}) and Tomas \cite{Tom75a,Tom75b}, their result was only for the surface measure on the sphere, $\mu=\sigma^{d-1}$. It was later on generalised by Mockenhaupt \cite{Moc00} who considered some specific measures built by Salem \cite{Sal51}. Although specific, his proof already contained all the ideas needed for the proof of Mitsis \cite{Mit02} which regards general singular measures satisfying polynomial Fourier decay and a Frostman condition. More recently still, Bak--Seeger \cite{BS11} proved that the range given by Mitsis could be improved to obtain the endpoint estimate as well. Their theorems consider the case $r=2$ and are as follows: If $\alpha>0$ and $\beta>0$ are such that
\begin{equation}\label{eq:FrostmanExp}
  \mu(B(x,r)) \lesssim r^\alpha
\end{equation}
for all $x \in \rd$ and $r>0$, and 
\begin{equation}\label{eq:FourierDecay}
  \big|\widehat{\mu}(\xi) \big|^2 \lesssim |\xi|^{-\beta} 
\end{equation}
for all $\xi \in \rd$, then for all $f\in L^2(\mu)$, the estimate
\begin{equation}\label{eq:restrictionr2}
    \|\widehat{f\mu}\|_{L^p(\rd)}\lesssim \|f\|_{L^2(\mu)}
\end{equation}
holds for all
\begin{equation}\label{eq:SteinTomas}
p \geq 2+4 \frac{d-\alpha}{\beta}.
\end{equation}
This   estimate is a cornerstone of Fourier restriction theory and is remarkable in the sense that a non-trivial range can be produced from only basic information about the measure, namely the Frostman condition \eqref{eq:FrostmanExp} and the Fourier decay condition \eqref{eq:FourierDecay}.

Recall that the corresponding restriction estimate for the sphere is known to be sharp by the Knapp example. Here  the idea is to build a function that captures the flatness of the sphere by taking it to be the indicator function of a small cap. In the fractal case, this flatness can be replaced by  considering arithmetic structure within the fractal. In \cite{HL13,HL16,Che16} the general idea is to choose an indicator function on a neighbourhood of an arithmetic progression. More recently, in \cite{FHR25,FHR25+} the authors consider a Salem set defined by Kaufmann \cite{Kau81}, which also contains long arithmetic progressions, showing that \eqref{eq:SteinTomas} is sharp in all dimensions $d\geq1$ for all $0<\alpha,\beta<d$. However, these measures that show sharpness the Mockenhaupt--Mitsis--Bak--Seeger estimate are all very regular, being supported on Salem sets, or on sets containing long arithmetic progressions (see Proposition~\ref{prop:LqdimensionsNorm}). 


\subsection{The $L^q$-dimensions and beyond Stein--Tomas}

The Frostman condition \eqref{eq:FrostmanExp} asks for uniform control  on the measure of balls.  The supremum of $\alpha$ for which \eqref{eq:FrostmanExp} holds is known as the Frostman dimension of $\mu$ and is an important quantity describing the fine scale geometry of the measure.  For many of  the classic examples in Fourier restriction, including surface measures on smooth manifolds, much more can be said.  In fact, \eqref{eq:FrostmanExp} can be replaced by
 \[
 \mu(B(x,r)) \approx  r^\alpha
 \]
 where $\alpha$ is the dimension of the support, for example the dimension of the manifold. That is, all balls carry roughly the same mass and so the measure is very uniformly distributed.  This is part of the reason why the Stein--Tomas estimate works so well for these examples. However, in general---and especially for measures exhibiting multifractality---we should not expect such uniformity.  There may be wild fluctuations in $\mu(B(x,r))$  and the Frostman dimension will not see these fluctuations; it only sees the extreme cases. The $L^q$-dimensions, defined later in Section \ref{sec:dimensions}, provide a more nuanced description of the local fluctuations of $\mu$ and are a central concept in fractal geometry, see for example \cite{falconer, HK97, Ols95, Pes93}, and play a crucial role in multifractal analysis and the theory of large deviations.  Roughly speaking, they describe the $q$-moments of the measure and recover the Frostman dimension at $q=\infty$ and many other notable notions of fractal dimension at other specific values.  As such, it is very natural to ask if the $L^q$-dimensions can be used to prove Stein--Tomas style restriction estimates and it is perhaps surprising that the marriage of fundamental concepts in harmonic analysis and fractal geometry seems to have been overlooked until now.   In this paper we establish such estimates (see Corollary \ref{cor:casetheta0}) and explore conditions under which they strictly improve on the Stein--Tomas estimate \eqref{eq:SteinTomas}.  We show that this happens quite often for measures exhibiting multifractal behaviour including certain Mandelbrot cascade measures (see Theorem \ref{cascade}), which are a fundamental model in probabilistic fractal geometry and a basic model for fluid turbulence. On the way to proving our main theorem, we are led to a novel characterisation of the $L^q$-dimensions in terms of certain convolution norms and this may be interesting in its own right (see Proposition \ref{prop:LqdimensionsNorm}). 
 
Just as the Frostman condition \eqref{eq:FrostmanExp} describes uniform control of the measure of small balls, the Fourier decay condition \eqref{eq:FourierDecay} asks for uniform control on the decay of the Fourier transform with the optimal $\beta$ defining the Fourier dimension of the measure.  The Fourier spectrum, defined in Section \ref{sec:FourierSpectrum} is a continuum of dimensions which provide a more nuanced description of the decay of the Fourier transform and recovers the Fourier dimension at one endpoint.  We also incorporate this into our restriction estimate (see Theorem \ref{thm:mainthm}) and setting $q=\infty$ this recovers, and provides a new proof of, the restriction theorem based on the Fourier spectrum established in \cite{CFdO25}.

\section{Preliminaries}\label{sec:prelim}

\subsection{The $L^q$-dimensions}\label{sec:dimensions}


Throughout the article we will work with non-zero, finite, compactly supported, Borel measures on $\rd$. We now give a description of a family of dimensions known as the $L^q$-dimensions. For more details we refer the readers to \cite{falconer, Ols95}.

Given $0<\alpha\leq d$, recall  a  measure $\mu$ satisfies the Frostman condition \eqref{eq:FrostmanExp} with exponent $\alpha$ if for all $r>0$ and $x\in\rd$, $\mu( B(x,r) )\lesssim r^\alpha$. If a measure $\mu$ satisfies this condition, its $\beta$-energy is finite for all $\beta<\alpha$, that is, 
\begin{equation*}
    I_{\beta}(\mu)\coloneqq \iint|x-y|^{-\beta}\,d\mu(x)\,d\mu(y)< \infty.
\end{equation*}
We define the Frostman and Sobolev dimensions of a measure as
\begin{equation*}
    \frd\mu = \sup\big\{ \alpha\in\R : \mu( B(x,r) )\lesssim r^\alpha,~~\forall r>0,x\in\rd \big\};
\end{equation*}
and
\begin{equation*}
    \sd\mu = \sup\{ \beta\in\R : \int_{\rd}\big| \widehat{\mu}(\xi) \big|^{2} |\xi|^{\beta-d}\,d\xi<\infty \},
\end{equation*}
respectively, noting that the above observation yields $\frd\mu\leq \sd\mu$.  Moreover, when $\mu$ is supported on a null set $ \sd\mu \leq d$, in which case the Sobolev dimension is also called the energy dimension, the correlation dimension, or the $L^2$ dimension. Further,  the Hausdorff dimension of a Borel set $X\subseteq\rd$ is the supremum of $\min\{ d,\sd\mu \}$, taken over all $\mu$ supported on $X$.

The Frostman (or $L^\infty$) dimension and Sobolev (or $L^2$) dimension are part of a more general family of dimensions that capture local fluctuations  of $\mu$. Given a compactly supported Borel probability measure $\mu$ on $\mathbb{R}^d$ and $q \geq 0$, the (lower) $L^q$-dimension of $\mu$ is defined by
\begin{equation}\label{eq:defLq}
  \tau_\mu(q) = \liminf_{r \to 0} \frac{\log \sum_{Q \in \mathcal{Q}_r} \mu(Q)^q}{(q-1)\log r}\text{ for }q\neq1;\text{ and } \tau_{\mu}(1) = \liminf_{r\to0}\frac{\sum_{Q\in\mathcal{Q}_{r}}\mu(Q)\log\mu(Q)}{\log r}
\end{equation}
where $\mathcal{Q}_r$ is the collection of $r$-cubes in an $r$-mesh oriented with the coordinate axes.  The corresponding upper $L^q$-dimensions are defined by replacing the $\liminf$ with   $\limsup$, but we will be more concerned with the lower version. The $L^q$-dimensions are non-increasing in $q\geq0$ and continuous except perhaps at $q=1$.  In many cases of interest the upper and lower $L^q$-dimensions coincide and are continuous at $q=1$, see \cite{feng,peressolomyak}.  There are many equivalent definitions of the $L^q$-dimensions, see e.g.~\cite{HK97,Pes93} and it is also possible to extend the definition to include negative $q$, but we will not use that here. In Proposition~\ref{prop:LqdimensionsNorm} we provide  a characterisation in terms of $L^q$ norms of convolutions of the measure with a Riesz kernel, which is particularly suited to restriction estimates.

The $L^q$-dimensions encode a lot of useful geometric information about $\mu$, including recovering several well-known notions of dimension as special cases, see \cite{ngai}.  For example, we have already seen
\[
\sd \mu = \tau_\mu(2)
\]
\[
\dim_\textup{Fr} \mu = \lim_{q \to \infty} \tau_\mu(q) =: \tau_\mu(\infty),
\]
but we also have
\[
\lbd \spt(\mu) = \tau_\mu(0)
\]
where $\lbd$ denotes the lower box dimension and $\tau_\mu(1)$ is the lower entropy dimension and, provided $\tau_\mu$ is continuous at 1, then   
\[
\tau_\mu(1) = \hd \mu.
\]
If  $\tau_\mu$ is continuous and coincides with the upper $L^q$-dimensions at $q=1$,  then $\mu$ is exact dimensional and $\tau_\mu(1) $ also gives the upper Hausdorff dimension and the (upper and lower) packing dimension of $\mu$. Further, the Legendre transform of $\mu$ is an upper bound for the fine and coarse multifractal spectra of $\mu$ and for measures satisfying the so-called multifractal formalism the Legendre transform of $\tau_\mu$ gives the multifractal spectra precisely, see \cite{Ols95}. As such, the $L^q$-dimensions are a central object of study in fractal geometry and have received sustained attention in the literature, see for example \cite{barralfeng,almost,shmerkin}. They are widely used to describe the multifractal behaviour of measures, especially useful for measures with widely varying intensity, see for example Figure~\ref{fig:multifractal}. Typical examples of such measures are the self-similar measures on middle third Cantor set and the Sierpi\'nski carpet built with not all-equal probabilities, Mandelbrot cascades (see Section~\ref{sec:cascade}), and many others.  Until recently, the $L^q$-dimensions have not received much attention in the harmonic analysis literature, perhaps because the classical examples (such as surface measures) tend to be uniformly distributed.  That said, the $L^q$-dimensions have recently seen applications in Fourier restriction theory in the multilinear setting, see \cite{OdO26+}.  

 Given $\alpha>\tau_{\mu}(\infty)$ we can always find a cube $Q'\in\mathcal{Q}_{r}$ with $r$ arbitrarily small such that $\mu(Q')\gtrsim r^\alpha$. Therefore, for $q>1$ and a sequence of  $r \to 0$, 
\begin{equation*}
    \frac{\log \sum_{Q \in \mathcal{Q}_r} \mu(Q)^q}{(q-1)\log r} \leq \frac{\log\mu(Q')^q}{(q-1)\log r} \leq \frac{q\alpha}{(q-1)}.
\end{equation*}
Letting $\alpha\to\tau_{\mu}(\infty)$ and $r\to0$ we obtain the following elementary  bound:  
\begin{equation}\label{eq:boundLqdimensions}
    \tau_{\mu}(q) \leq \frac{q}{q-1}\tau_{\mu}(\infty).
\end{equation}
However, we show that those measures for which our main theorem does not improve the Stein--Tomas estimate \eqref{eq:SteinTomas} must satisfy a more restrictive  upper bound than \eqref{eq:boundLqdimensions}, see  Corollary~\ref{improve} and the discussion thereafter.

\begin{figure}[H]
  \includegraphics[scale=0.45]{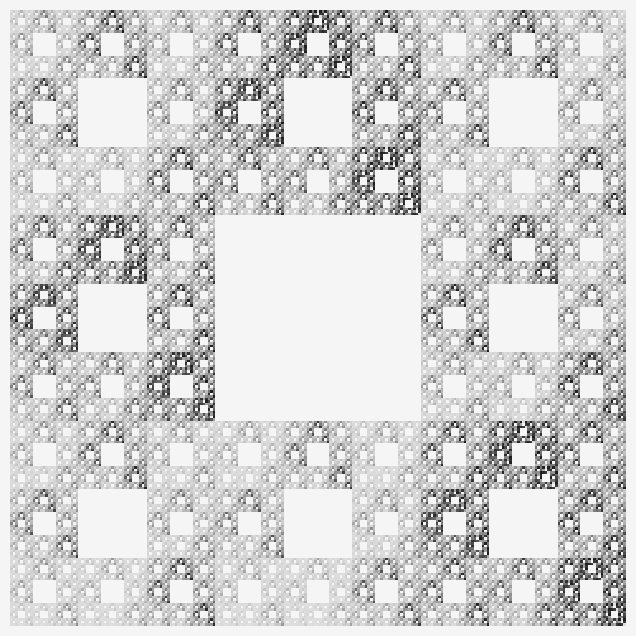}
  \caption{A multifractal measure supported on the Sierpi\'nski carpet. The dark regions represent high concentration of mass and the light regions low concentration of mass.}\label{fig:multifractal}
\end{figure}

\subsection{Dimensions given by the Fourier transform}\label{sec:FourierSpectrum}

Given a function $f:\rd\to\C$, $f\in L^1(\rd)$, its Fourier transform is defined as
\begin{equation*}
    \widehat{f}(\xi) = \int_{\rd} e^{-2\pi i \xi\cdot x}f(x)\,dx,
\end{equation*}
and this operator can be easily extended to functions in $L^p$ for $p\in[1,2]$.  We refer the reader to \cite{grafakos,Mat15,Ste93} for more background on Fourier analysis in general.  The Fourier transform of a measure $\mu$ is
\begin{equation*}
    \widehat{\mu}(\xi) = \int_{\rd}e^{-2\pi i \xi\cdot x}\,d\mu(x).
\end{equation*}
For $0< \beta< d$, by Parseval's theorem and the fact that the Riesz kernel $\kappa_{\beta}(\xi)\coloneqq |\xi|^{-\beta}$ satisfies that $\widehat{\kappa_{\beta}}(\xi) = \kappa_{d-\beta}(\xi)$ in the distributional sense,
\begin{equation}\label{eq:senergy}
    I_{\beta}(\mu)\approx_{d,\beta} \int_{\rd}\big| \widehat{\mu}(\xi) \big|^{2} |\xi|^{\beta-d}\,d\xi.
\end{equation}
From this it is straightforward to see that if $\big| \widehat{\mu}(\xi) \big|\lesssim |\xi|^{-\frac{\beta}{2}}$ decays polynomially, then $I_{\beta'}(\mu)$ will be finite for all $\beta'<\beta$. The supremum of the polynomial rate of decay of $\big| \widehat{\mu}(\xi)\big|$ is the Fourier dimension of $\mu$, that is,
\begin{equation*}
    \fd\mu = \sup\big\{ \beta\in\R : \sup_{\xi\in\rd}\big| \widehat{\mu}(\xi) \big|^2 |\xi|^{\beta}<\infty \big\}.
\end{equation*}
This connection between the energies $I_\beta(\mu)$ and Fourier decay opens up a rich interplay between fractal geometry and Fourier analysis.  In \cite{Fra24} the author generalised the $\beta$-energies to arrive at a more comprehensive  description  of the Fourier analytic behaviour of $\mu$. For $\theta\in(0,1]$, define the $(\beta,\theta)$-energy of a measure $\mu$ as
\begin{equation*}
    \J_{\beta,\theta}(\mu)\coloneqq \bigg(\int_{\rd}\big| \widehat{\mu}(\xi) \big|^{\frac{2}{\theta}}|\xi|^{\frac{\beta}{\theta}-d}\,d\xi\bigg)^\theta
\end{equation*}
and for $\theta = 0$,
\begin{equation*}
    \J_{\beta,0}(\mu) = \sup_{\xi\in\rd}\big| \widehat{\mu}(\xi) \big|^{2}|\xi|^{\beta}.
\end{equation*}
Then the Fourier spectrum of $\mu$ at $\theta$ is
\begin{equation*}
    \fs\mu = \sup\{ \beta\in\rd : \J_{\beta,\theta}(\mu)<\infty \}.
\end{equation*}
There is very little connection between the Fourier spectrum and the $L^q$-dimensions because they measure different things.  However, they both witness the $L^2$ dimension
\[
\tau_\mu(2) = \dim_\textrm{F}^1 \mu 
\]
in the case when $\mu$ is supported on a null set and it is also useful to keep in mind the bound
\[
\fd \mu \leq 2 \tau_\mu(\infty) 
\]
which always holds and relates the Frostman and Fourier dimensions.  This estimate is due to Mitsis \cite{Mit02}.

The Fourier spectrum has seen applications in the restriction problem \cite{CFdO25, FdO26+}, the Falconer distance problem \cite{Fra24, pham},  the dimension theory of orthogonal projections \cite{FdO24}, and the Fourier dimension of fractional Brownian motion \cite{LL25+}, among other problems.

\section{Main results} \label{sec:main}


In the following two theorems we obtain a restriction estimate for fractal measures in terms of their Fourier spectrum and $L^q$-dimensions; see also Corollary~\ref{cor:casetheta0} for a result in terms of the Fourier dimension and $L^q$-dimensions. These results unify and significantly extend key theorems from the literature.  For example,  setting $\theta = 0$ and $q= \infty$ we recover the Mockenhaupt--Mitsis--Bak--Seeger estimate \eqref{eq:SteinTomas}, and setting $\theta \in[0,1]$ and $q=\infty$ we recover, and provide a novel proof of, \cite[Theorem~3.1]{CFdO25}. 

Our proof proceeds via complex interpolation, making it closer to the proof of Stein (see e.g. \cite{Tom75b}).  In \cite{BS11}, Bak--Seeger remarked that it was not clear how to extend the complex interpolation approach to general fractal measures satisfying \eqref{eq:FrostmanExp} and \eqref{eq:FourierDecay} (that is, our setting with $q=\infty$ and $\theta=0$), see the discussion preceding \cite[Theorem 1.1]{BS11}. We propose a solution to  Bak and Seeger's problem by replacing  the \emph{ball} condition \eqref{eq:FrostmanExp} with the  \emph{potential} condition
  \begin{equation} \label{potentialcondition}
      \sup_{x\in\rd}\int \frac{d\mu(y)}{|x-y|^{\alpha}} < \infty.
  \end{equation}
In doing so we are able to prove an endpoint result for Frostman dimension via complex interpolation.  In particular, the Frostman dimension is equal to either the supremum of $\alpha$ satisfying \eqref{eq:FrostmanExp} or  \eqref{potentialcondition}, see Proposition \ref{prop:LqdimensionsNorm}.  In general \eqref{eq:FrostmanExp} and  \eqref{potentialcondition} are not equivalent, but \eqref{potentialcondition} does imply \eqref{eq:FrostmanExp}.

We first state the most general case which includes the endpoint estimate. We will see below in Proposition \ref{prop:LqdimensionsNorm} that condition \eqref{lqnormcondition} precisely characterises the $L^q$-dimensions of $\mu$, see Theorem \ref{thm:mainthm} below for a formulation which explicitly uses the $L^q$-dimensions.
\begin{thm}\label{thm:mainthmEndpoint}
  Let $0<\alpha<d$, $\theta\in[0,1]$, $q>1$, $\beta>d\theta$, and let $\mu$ be a non-zero, finite, compactly supported, Borel measure on $\rd$, such that  
  \begin{equation} \label{lqnormcondition}
      \int\bigg( \int \frac{d\mu(y)}{|x-y|^{\frac{d + \alpha(q-1)}{q}}} \bigg)^q \,dx < \infty,
  \end{equation}
  or, for $q=\infty$,
  \begin{equation*}
      \sup_{x\in\rd}\int \frac{d\mu(y)}{|x-y|^{\alpha}} < \infty,
  \end{equation*}
  and
  \begin{equation*}
    \mathcal{J}_{\beta,\theta}(\mu)<\infty.
  \end{equation*}
  Let
  \begin{equation*}
      p \geq 2\, \frac{2(q-1)(d-\alpha) + q(\beta -\theta d) }{(q - 1)(\beta - \theta\alpha)}.
  \end{equation*}
  Then, for all $f\in L^{2}(\mu)$,
  \begin{equation*}
      \| \widehat{f\mu} \|_{L^{p}(\rd)} \lesssim \| f \|_{L^2(\mu)}.
  \end{equation*}
  Equivalently, let
  \begin{equation*}
      1\leq p'\leq  \frac{4 (q-1)(d-\alpha) + 2 q(\beta -\theta d)}{4 (q-1)(d-\alpha)  + 2 q(\beta - \theta d) - (q-1) (\beta - \theta\alpha)}.
  \end{equation*}
  Then, for all $f\in L^{p'}(\rd)$,
  \begin{equation*}
      \| \widehat{f\,} \|_{L^2(\mu)} \lesssim \| f \|_{L^{p'}(\rd)}.
  \end{equation*}
\end{thm}

If we wish to formulate the previous theorem in terms of the $L^q$-dimensions $\tau_q(\mu)$ and Fourier spectrum $\fs \mu$, then we surrender the endpoint estimate and obtain the following.  

\begin{thm}\label{thm:mainthm}
  Let $\mu$ be a non-zero, finite, compactly supported, Borel measure on $\rd$. If for some $q\in(1,\infty]$ and $\theta \in [0,1]$ such that $\fs\mu >d\theta$,
  \begin{equation*}
      p> 2 \,  \frac{2(q-1)(d-\tau_{\mu}(q)) + q(\fs\mu -\theta d) }{(q - 1)(\fs\mu - \theta\tau_{\mu}(q))},
  \end{equation*}
  then, for all $f\in L^{2}(\mu)$,
  \begin{equation*}
      \| \widehat{f\mu} \|_{L^p(\rd)} \lesssim \| f \|_{L^2(\mu)}.
  \end{equation*}
  Equivalently, if 
  \begin{equation*}
      1\leq p' <\frac{4 (q-1)(d-\tau_{\mu}(q)) + 2 q(\fs\mu -\theta d)}{4 (q-1)(d-\tau_{\mu}(q))  + 2 q(\fs\mu - \theta d) - (q-1) (\fs\mu - \theta\tau_{\mu}(q))},
  \end{equation*}
  then, for all $f\in L^{p'}(\rd)$,
  \begin{equation*}
      \| \widehat{f\,} \|_{L^2(\mu)} \lesssim \| f \|_{L^{p'}(\rd)}.
  \end{equation*}
\end{thm}

\begin{proof}
By definition of the Fourier spectrum, for all $\fs \mu> \beta> d\theta$,  
  \begin{equation*}
      \int_{\rd}\big| \widehat{\mu}(\xi) \big|^{\frac{2}{\theta}} |\xi|^{\frac{\beta}{\theta} - d} \,d\xi <\infty.
  \end{equation*}
  Moreover, by appealing to Proposition \ref{prop:LqdimensionsNorm} below,   for all $ \tau_\mu(q)> \alpha$,  
    \begin{equation*}
      \int\bigg( \int \frac{d\mu(y)}{|x-y|^{\frac{d + \alpha(q-1)}{q}}} \bigg)^q \,dx < \infty.
  \end{equation*}
  The desired estimates (without the endpoint) then follow immediately from Theorem \ref{thm:mainthmEndpoint}.
\end{proof}

In the previous theorem one can optimise over the parameters $q$ and $\theta$ to obtain the best range for $p$ by taking it to be
\begin{equation*}
    p >2\inf\limits_{\substack{\theta\in[0,1] \\ \fs\mu>d\theta\\q>1}} \frac{2(q-1)(d-\tau_{\mu}(q)) + q(\fs\mu -\theta d) }{(q - 1)(\fs\mu - \theta\tau_{\mu}(q))},
\end{equation*}
and analogously for $p'$ taking the supremum.

Setting $\theta = 0$ in Theorem~\ref{thm:mainthm} for measures with polynomial Fourier decay, we obtain the following restriction estimate involving the Fourier and $L^q$-dimensions only.  This result is aesthetically appealing, simple to state, and clearly   improves    the  Stein--Tomas estimate  \eqref{eq:SteinTomas}, see Corollary \ref{improve} below.  In particular, we recover Stein--Tomas when $q=\infty$ and recover \cite[Corollary~6.1]{CFdO25} when $q=2$.

\begin{cor}\label{cor:casetheta0}
  Let $\mu$ be a non-zero, finite, compactly supported, Borel measure on $\rd$ with $\fd\mu >0$. If, for some $q\in(1,\infty]$
  \begin{equation*}
      p>  \frac{2q}{q-1} + \frac{4(d-\tau_{\mu}(q))}{\fd\mu},
  \end{equation*}
  then for all $f\in L^2(\mu)$,
  \begin{equation*}
      \| \widehat{f\mu} \|_{L^p(\rd)} \lesssim \| f \|_{L^2(\mu)}.
  \end{equation*}
  Equivalently, if
  \begin{equation*}
      1\leq p' <2 \frac{2(q-1)( d- \tau_{\mu}(q)) + q\fd\mu}{4(q-1)(d-\tau_{\mu}(q)) + (q+1)\fd\mu},
  \end{equation*}
  then for all $f\in L^{p'}(\rd)$,
  \begin{equation*}
      \| \widehat{f\,} \|_{L^2(\mu)} \lesssim \| f \|_{L^{p'}(\rd)}.
  \end{equation*}
\end{cor}

Next we consider precisely when Corollary \ref{cor:casetheta0} improves the Stein--Tomas estimate  \eqref{eq:SteinTomas}.

\begin{cor} \label{improve}
Let $\mu$ be a   non-zero, finite, compactly supported, Borel measure on $\rd$ with $\fd \mu >0$. The restriction estimate from  Corollary \ref{cor:casetheta0} is strictly better than the Stein--Tomas estimate  \eqref{eq:SteinTomas} whenever 
\begin{equation} \label{improveneed}
\tau_{\mu}(q) > \tau_\mu(\infty)+ \frac{\fd \mu }{2(q-1)}
\end{equation}
for some $q>1$.
\end{cor}
The condition \eqref{improveneed} is rather mild for measures exhibiting multifractal behaviour.  In particular, the Fourier dimension is often very small (even zero for many self-similar measures) and in the multifractal setting $\tau_\mu(q)>\tau_\mu(\infty)$ holds.   The elementary upper bound \eqref{eq:boundLqdimensions} giving 
 \[
  \tau_{\mu}(q) \leq   \frac{q \tau_\mu(\infty)}{q-1}
 \]
is strictly larger than the right hand side of  \eqref{improveneed} whenever $\fd \mu < 2 \tau_\mu(\infty)$ and so there is space for improvement apart from in the extreme case $\fd \mu = 2 \tau_\mu(\infty)$, recalling that $\fd \mu \leq 2 \tau_\mu(\infty)$ always holds by the estimate due to Mitsis \cite{Mit02}.

A simple consequence of Corollary~\ref{improve} is that any measure for which our estimates do not improve on Stein--Tomas must satisfy 
  \begin{equation} \label{constrain}
     \tau_{\mu}(q) \leq \tau_{\mu}(\infty) +  \frac{\fd\mu}{2(q-1)}
  \end{equation}
  for all $q>1$.  In particular, this applies to  measures for which the Stein--Tomas estimate is sharp, e.g.~those defined in \cite{HL13, HL16,FHR25,FHR25+}.  One could also formulate bounds using Theorem~\ref{thm:mainthm} instead of Corollary~\ref{cor:casetheta0} but we leave these to the reader. The bound \eqref{constrain}  improves upon the elementary bound  \eqref{eq:boundLqdimensions} whenever $\fd \mu < 2 \tau_\mu(\infty)$ and recovers \eqref{eq:boundLqdimensions} in the extreme case when $\fd \mu = 2 \tau_\mu(\infty)$.
  
Given Corollary \ref{cor:casetheta0} and \eqref{constrain}, we are able to establish   new information  for \emph{all} measures.  That is, for all measures at least one of the following must hold:\begin{enumerate}
\item   Corollary \ref{cor:casetheta0} strictly improves the Stein--Tomas restriction estimate,
\item \eqref{constrain} strictly improves the bound \eqref{eq:boundLqdimensions} for $\tau_\mu(q)$  for all $q>1$,
\item $\fd \mu = 2 \tau_\mu(\infty)$.
\end{enumerate}

\section{Obtaining the $L^q$-dimensions from a convolution norm}\label{sec:LqdimensionsNorm}

To leverage the information provided by the $L^q$-dimensions, we use an alternative definition in terms of $L^q$ norms of certain convolutions, see Section~\ref{sec:dimensions} for the usual definition. There are many alternative formulations of the $L^q$-dimensions in the literature (see e.g.~\cite{Pes93}), but the one we present here seems to be new.  It is particularly suited to convolution problems and may be interesting in its own right.  Note that the energy we use in the formulation below is efficiently expressed as the $L^q$ norm of the convolution with an appropriate Riesz kernel, that is, 
 \[
     \| \kappa_{s_1}*\mu \|_{L^{q}(\rd)}^{q} = 
 \int\bigg( \int \frac{d\mu(y)}{|x-y|^{\frac{d + \alpha(q-1)}{q}}} \bigg)^q \,dx,
\]
for  $s_{1} := \frac{d+(q-1)\alpha}{q}$. 

\begin{prop}\label{prop:LqdimensionsNorm}
 Let $\mu$ be a compactly supported Borel probability measure on $\mathbb{R}^d$. Then for all $q\in(1,\infty)$
 \[
 \tau_\mu(q) = \sup\left\{\alpha \geq 0 : \int_{\rd} \Bigg( \int \frac{d\mu(y)}{|x-y|^{\frac{d + (q-1)\alpha}{q}}}\Bigg)^q \, dx < \infty \right\},
 \]
 and for $q=\infty$,
 \begin{equation*}
     \tau_{\mu}(\infty) = \sup\left\{\alpha \geq 0 : \sup_{x\in\rd}  \int \frac{d\mu(y)}{|x-y|^{\alpha}} < \infty \right\}.
 \end{equation*}
\end{prop}

\begin{proof}
First define
\[
T(q) = \sup\left\{\alpha \geq 0 : \int_{\rd} \left( \int \frac{d\mu(y)}{|x-y|^{\frac{d + (q-1)\alpha}{q}}}\right)^q \, dx < \infty \right\}.
\]
Let $\alpha<T(q)$ and $r\in (0,1)$.  Write $Q_x$ for the $r$-cube from $\mathcal{Q}_r$ containing a point $x$.  If this is not uniquely defined, then choose one arbitrarily.  Then, writing $t=\frac{d + (q-1)\alpha}{q}$,
\begin{align*}
r^{-\alpha(q-1)}\sum_{Q \in \mathcal{Q}_r} \mu(Q)^q & = r^{-\alpha(q-1)}  \int_{\rd} \mu(Q_x)^q r^{-d} \, dx \\
& =   \int_{\rd} \left( \frac{\mu(Q_x)}{r^t}\right)^q  \, dx \\
& =   \int_{\rd} \left( \int_{Q_x}\frac{d \mu(y)}{r^t}\right)^q  \, dx \\
& \lesssim   \int_{\rd} \left( \int_{Q_x}\frac{d \mu(y)}{|x-y|^t}\right)^q  \, dx \\
& \leq   \int_{\rd} \left( \int\frac{d \mu(y)}{|x-y|^t}\right)^q  \, dx \\
&<\infty.
\end{align*}
Since this holds for all $r>0$, this proves that $\tau_\mu(q) \geq \alpha$ and therefore $\tau_\mu(q) \geq T(q)$.  

We now prove the other direction. Let $\alpha<\tau_\mu(q)$,  $r\in (0,1)$ and first observe that
\begin{equation} \label{lqest1}
\int_{\rd} \mu\big(B(x,r)\big)^q \, dx \lesssim \int_{\rd} \mu(Q_x)^q \, dx  = r^d\sum_{Q \in \mathcal{Q}_r} \mu(Q)^q \lesssim r^{d+\alpha(q-1)}.
\end{equation}
In the above we are not claiming that for all $x \in \rd$, the pointwise estimate $\mu\big(B(x,r)\big)  \lesssim   \mu(Q_x)$ holds.  Indeed, this is clearly false in general in cases where $\mu$ is not doubling.  However,  when one averages over $x$, the estimate is recovered.  More precisely,  write $\mathcal{Q}_r^x \subseteq \mathcal{Q}_r$ to be the collection of $3^d$ cubes either equal to or adjacent to $Q_x$.  Then 
\[
B(x,r) \subseteq \bigcup_{Q \in \mathcal{Q}_r^x} Q
\]
and
\begin{align*}
\int_{\rd} \mu\big(B(x,r)\big)^q \, dx &\leq 3^d \int_{\rd} \max_{Q \in \mathcal{Q}_r^x}\mu\big(Q\big)^q \, dx\\
&\leq 3^d\sum_{Q'} \mu(Q')^q \, \mathcal{L}^d\Big(\Big\{x : Q'= \argmax\limits_{Q \in \mathcal{Q}_r^x}\mu\big(Q\big) \Big\}\Big) \\
&\leq 3^d \sum_{Q'} \mu(Q')^q \, 3^d \, r^d\\
&= 9^d \int_{\rd} \mu(Q_x)^q \, dx,
\end{align*}
as required. Next, let $u <t=\frac{d + (q-1)\alpha}{q}$, and observe that, for a fixed $x\in\rd$,
\begin{align} 
\int\frac{d \mu(y)}{|x-y|^u} &= \int_{|x-y|>1} \frac{d \mu(y)}{|x-y|^u} \ + \  \sum_{k=0}^\infty \int_{2^{-k-1} < |x-y| \leq 2^{-k}}\frac{d \mu(y)}{|x-y|^u} \nonumber \\
&\lesssim 1 + \sum_{k=0}^\infty  2^{ku} \mu(B(x,2^{-k})).  \label{lqest2}
\end{align}
Finally, applying the triangle inequality for $L^q$ norms, \eqref{lqest1}, and \eqref{lqest2},
\begin{align*} 
\left(\int_{\rd} \left( \int\frac{d \mu(y)}{|x-y|^u}\right)^q \, dx\right)^{1/q} &\lesssim  1 + \sum_{k=0}^\infty  2^{ku} \bigg(\int_{\rd} \mu\big(B(x,2^{-k})\big)^q \, dx\bigg)^{1/q} \\
&\lesssim  1 + \sum_{k=0}^\infty  2^{ku} \Big(2^{-k(d+\alpha(q-1))}\Big)^{1/q}\\
&=1 + \sum_{k=0}^\infty  2^{k(u-t)} \\
&<\infty.
\end{align*}
Letting  $u \to t$, we get that $T(q) \geq \alpha$ and therefore $T(q) \geq \tau_\mu(q)$, completing the proof for $1<q<\infty$.

The case $q=\infty$ is easier to prove. Define
\begin{equation*}
    T(\infty) = \sup\left\{\alpha \geq 0 : \sup_{x\in\rd}  \int \frac{d\mu(y)}{|x-y|^{\alpha}} < \infty \right\}
\end{equation*}
and let $\alpha<T(\infty)$. For  $y\in B(x,r)$  we  have $1\leq r^{\alpha}|x-y|^{-\alpha}$ and, therefore,
\begin{equation*}
    \mu\big( B(x,r) \big) = \int_{B(x,r)}d\mu(y) \leq r^\alpha \int \frac{d\mu(y)}{|x-y|^{\alpha}} \lesssim r^{\alpha}.
\end{equation*}
This proves that $\tau_{\mu}(\infty)\geq \alpha$ and therefore $\tau_{\mu}(\infty)\geq T(\infty)$.

To prove the other direction let $u<\alpha<\tau_{\mu}(\infty)$. By the same argument as in \eqref{lqest2}, for all $x\in \rd$
\begin{equation*}
    \int \frac{d\mu(y)}{|x-y|^{u}} \lesssim 1 + \sum_{k=0}^{\infty} 2^{k(u-\alpha)} <\infty.
\end{equation*}
Letting $u\to \alpha$ we get that $T(\infty)\geq \alpha$ and thus $T(\infty)\geq\tau_{\mu}(\infty)$, as required.
\end{proof}

\section{Proof of Theorem \ref{thm:mainthmEndpoint}}

To prove the main theorems we will use the following lemma, known as Stein's complex interpolation theorem. We refer the reader to  \cite[Chapter~IX.1.2.5]{Ste93} for more background on this result.
\begin{lma}[Stein's complex interpolation theorem]\label{lma:SteinComplex}
  For $0\leq \Re(z)\leq 1$, let $T_{z}$ be a family of linear operators depending analytically on $z$. Let $1 \leq p_{0},p_{1}\leq \infty$, $1\leq q_{0},q_{1}\leq \infty$ and $M_{0},M_{1}>0$ be such that for $j=0,1$, and $\Re(z) = j$,
  \begin{equation*}
      \| T_{z}(f) \|_{L^{q_{j}}(\rd)} \leq M_{j}\| f \|_{L^{p_{j}}(\rd)},
  \end{equation*}
  where the constants $M_{j}$ grow slower than the double exponential on $\Im(z)$. Then for all $0\leq \lambda\leq 1$,
  \begin{equation*}
      \| T_{\lambda}(f) \|_{L^{q_{\lambda}}(\rd)} \leq M_{0}^{1-\lambda}M_{1}^{\lambda} \| f \|_{L^{p_{\lambda}}(\rd)},
  \end{equation*}
  where $\frac{1}{p_{\lambda}} = \frac{1-\lambda}{p_{0}} + \frac{\lambda}{p_{1}}$, and  $\frac{1}{q_{\lambda}} = \frac{1-\lambda}{q_{0}} + \frac{\lambda}{q_{1}}$.
\end{lma}

Now we proceed with the proof  of Theorem  \ref{thm:mainthmEndpoint}.  Let  $\theta\in[0,1]$, $q \in (1,\infty]$, $\beta>d\theta$, $0<\alpha<d$.  In the hypotheses of the theorem we assume that, for $s_{1} \coloneqq \frac{d+(q-1)\alpha}{q}$,
  \begin{equation} \label{lqassume}
     \| \kappa_{s_1}*\mu \|_{L^{q}(\rd)}^q= 
 \int\bigg( \int \frac{d\mu(y)}{|x-y|^{\frac{d + \alpha(q-1)}{q}}} \bigg)^q \,dx < \infty,
 \end{equation}
 where for $q=\infty$ we interpret  $s_{1}\coloneqq \alpha$ and assume
\begin{equation*}
    \|\kappa_{s_{1}}*\mu\|_{L^{\infty}(\rd)} = \sup_{x\in\rd} \int \frac{d\mu(y)}{|x-y|^{\alpha}} < \infty.
\end{equation*}
Then, for all $a,b\in\R$,
\begin{align*}
  \| \kappa_{a+ib}*\mu \|_{L^{q}(\rd)}^{q} &= \int_{\rd} \bigg| \int |x-y|^{-a} e^{-ib\log|x-y|} d\mu(y) \bigg|^q \,dx \\
  &\leq \int \bigg( \int \frac{d\mu(y)}{|x-y|^a} \bigg)^q \,dx\\
  & = \| \kappa_{a}*\mu \|_{L^{q}(\rd)}^{q}
\end{align*}
with an analogous estimate in the $q=\infty$ case.  In particular, for all $b\in\R$,
\begin{equation}\label{eq:kappaLq}
    \| \kappa_{s_{1}+ib}*\mu \|_{L^{q}(\rd)} \lesssim 1.
\end{equation}

Also,  in the hypotheses of the theorem we assume that
  \begin{equation}\label{fdassume}
   \big\| \reallywidehat{\kappa_{\frac{d(2-\theta)+\beta}{2}}*\mu} \big\|_{L^{2/\theta}(\rd)}^2 =    \J_{\beta,\theta}(\mu) <\infty.
  \end{equation}
Then, for all $a,b\in\R$,
\begin{align*}
    \| \reallywidehat{\kappa_{a + ib}*\mu} \|_{L^{2/\theta}(\rd)}^{2/\theta} &\leq \int_{\rd} \big| \widehat{\mu}(\xi) \big|^{\frac{2}{\theta}} |\xi|^{\frac{2(a-d)}{\theta}} \big| e^{i \frac{2}{\theta}b \log|\xi|} \big| \,d\xi\\
    & = \int_{\rd} \big| \widehat{\mu}(\xi) \big|^{\frac{2}{\theta}} |\xi|^{\frac{2(a-d)}{\theta}} \,d\xi \\
    &= \| \widehat{\kappa_{a}*\mu} \|_{L^{2/\theta}(\rd)}^{2/\theta}.
\end{align*}
Therefore,  for $s_{0}\coloneqq \frac{d(2-\theta) + \beta}{2}$, and all $b\in\R$,  
\begin{equation}\label{eq:kappaFS}
  \| \reallywidehat{\kappa_{s_{0} + ib}*\mu} \|_{L^{2/\theta}(\rd)}\lesssim 1,
\end{equation}
with the appropriate $L^\infty$ interpretation for $\theta=0$. 

For $z\in\C$ let $s(z) = s_{0}(1-z) + s_{1}z$ and define 
\[
T_{s(z)}f = \reallywidehat{\kappa_{s(z)}*\mu}*f.
\]
Let us first consider the case $z = 0 + ib$ for $b\in\R$, in which case $s(z) = s_{0}+ib(s_{1}-s_{0})$. By Young's inequality for convolutions and \eqref{eq:kappaFS}, if $1 + \frac{1}{r_{0}} = \frac{\theta}{2} + \frac{1}{p_{0}'}$, then
\begin{equation}\label{eq:interpolation0}
    \| T_{s(0 + ib)}f \|_{L^{r_{0}}(\rd)} \leq \| \reallywidehat{\kappa_{s_{0}+ib(s_{1}-s_{0})}*\mu} \|_{L^{2/\theta}(\rd)} \| f \|_{L^{p_{0}'}(\rd)} \lesssim \| f \|_{L^{p_{0}'}(\rd)}.
\end{equation}
Now consider $z = 1 + ib$ for $b\in\R$, in which case $s(z) = s_{1}+ib(s_{1}-s_{0})$. By the  Hausdorff--Young inequality and H\"older's inequality, if  $r_{1},p_{1}\geq2$ and $\frac{1}{q} + \frac{1}{r_{1}} = \frac{1}{p_{1}'}$,
\begin{align*}
    \| T_{s(z)}f \|_{L^{r_{1}}(\rd)} &= \| \reallywidehat{(\kappa_{s(z)}*\mu)\widecheck{f}} \|_{L^{r_{1}}(\rd)} \\
    &\leq \| (\kappa_{s(z)}*\mu)\widecheck{f} \|_{L^{r_{1}'}} \\
    &\leq \| \kappa_{s(z)}*\mu \|_{L^q(\rd)}\| \widecheck{f} \|_{L^{p_{1}}(\rd)} \\
    & \leq \| \kappa_{s(z)}*\mu \|_{L^q(\rd)}\| f \|_{L^{p_{1}'}(\rd)},\numberthis\label{eq:ineq2}
\end{align*}
which by \eqref{eq:kappaLq} gives
\begin{equation}\label{eq:interpolation1}
    \|T_{s(1 + ib)}f\|_{L^{r_{1}}(\rd)} \lesssim \|f\|_{L^{p_{1}'}(\rd)}.
\end{equation}
From \eqref{eq:interpolation0} and \eqref{eq:interpolation1} it follows from Stein's complex interpolation theorem, Lemma~\ref{lma:SteinComplex}, that for all $\lambda\in[0,1]$,
\begin{equation*}
    \| T_{s(\lambda)}f \|_{L^r(\rd)} \lesssim \| f \|_{L^{p'}(\rd)},
\end{equation*}
where
\begin{equation}\label{eq:equationsPR}
    \frac{1}{r} = \frac{1-\lambda}{r_{0}} + \frac{\lambda}{r_{1}} \quad \text{and}\quad \frac{1}{p'} = \frac{1-\lambda}{p_{0}'} + \frac{\lambda}{p_{1}'}.
\end{equation}
Let $\lambda = \frac{s_{0}-d}{s_{0}-s_{1}}$.  Observe that, since $\beta >d\theta$ and $\alpha<d$, we know $s_0>d>s_1$.  It follows that $\lambda\in(0,1)$.  Moreover,  $s(\lambda) = d$ and $T_{s(\lambda)}f = \widehat{\mu}*f$. Taking $r = p$, \eqref{eq:equationsPR} gives
\begin{equation*}
    p = 2\frac{q\beta - 2 \alpha (q - 1) - d (q\theta + 2 - 2 q)}{(\beta - \theta\alpha)(q - 1)},
\end{equation*}
which proves Theorem~\ref{thm:mainthmEndpoint}.

\section{Mandelbrot cascades}\label{sec:cascade}

Mandelbrot cascades are an important example in fractal geometry and probability theory, and were originally introduced as a fundamental model for fluid turbulence.  They are a family of random measures displaying stochastic self-similarity and have many interesting geometric properties.  In particular, they are often multifractal and the Fourier dimension was recently computed (independently) in \cite{suomalacascades,chenhan}.  These two properties make them ideal candidates for testing our theory (at least in the case $\theta=0$).  We briefly recall the definition.  Let $N \geq 2$ be an integer and $W \geq 0$ be a non-negative random variable with $\mathbb{E}(W) = 1$. For each integer $n \geq 1$, let $\mathcal{D}_n$ denote the collection of closed $N$-adic sub-cubes of $[0,1]^d$ of side-length $N^{-n}$ and to each $Q \in \cup_n \mathcal{D}_n$ associate an independent copy of $W$, denoted by $W_Q$.  Define a random density $\mu_n$ by 
\[
\mu_n(x) = \prod_{i=1}^n W_{Q_i}
\]
for $x \in Q_n \in \mathcal{D}_n$ and  $Q_j \in \mathcal{D}_j$ such that $Q_n \subseteq Q_j$ for each $i = 1, \dots, n$.  The measures associated to these densities converge weakly to a random limit $\mu$, which we call the Mandelbrot cascade.  The Mandelbrot cascade is called non-degenerate if $\mathbb{E} (W \log W) < d\log N$ and sub-exponential if for all $t >0$, $\mathbb{P}( W \geq t) \leq 2 \exp(-ct)$ for some $c$.  We quote two results from the literature in the non-degenerate sub-exponential cases; see \cite{Mol96, heurteaux} for the $L^q$-dimensions and \cite{suomalacascades,chenhan} for the Fourier dimension. For $p>0$ define
\begin{equation*}
    W_p:= \frac{W^p}{\mathbb{E}(W^p)},
\end{equation*}
and let $\lambda\in(0,\infty)$ be such that
\begin{equation*}
    \E(W_{\lambda}\log W_{\lambda}) = d\log N.
\end{equation*}
Then, almost surely conditioned on non-extinction for $q\neq1$:
\begin{equation*}
   \tau_{\mu}(q) = \begin{cases}
    d- \frac{\log \mathbb{E}(W^q)}{(q-1)\log N}\,, &\text{if }\E(W_q\log W_q) \leq d\log N;\\
    \frac{qd(\lambda-1)}{\lambda(q-1)} - \frac{q\log\E(W^\lambda)}{\lambda(q-1)\log N}\,, &\text{if }\E(W_{q}\log W_{q})>d\log N,
  \end{cases}
\end{equation*}
and for $q=1$
\begin{equation*}
    \tau_{\mu}(1) = 
    d - \frac{\mathbb{E}(W\log W)}{\log N}.
\end{equation*}
The Fourier dimension is given by
\[
\fd \mu = \min\{2, \tau_{\mu}(2)\}.
\]
In particular, if we are in the subcritical regime where $\mathbb{E}(W_2\log W_2) \leq d \log N$ and in the case $\tau_\mu(2) \leq 2$, then
\[
\tau_\mu(2) = \sd \mu = \fd \mu = \fs \mu =  d- \frac{\log \mathbb{E}(W^2)}{\log N}
\]
for all $\theta \in [0,1]$.

Applying our main result, Theorem~\ref{thm:mainthm}, in the case $\theta=0$ with the above in mind we get the following restriction estimate.

\begin{thm} \label{cascade}
Let $\mu$ be a non-degenerate, sub-exponential Mandelbrot cascade measure and $q \in (1, \infty]$ be arbitrary. Then, almost surely conditioned on non-extinction,  $\|\widehat{f\mu}\|_{L^p(\rd)}\lesssim \|f\|_{L^2(\mu)}$ for all $f \in L^2(\mu)$ whenever
\begin{equation*}
   p > \frac{2q}{q-1} +  \frac{4  \log \mathbb{E}(W^q)  }{ (q-1) (\log N) \min\{2, \tau_{\mu}(2)\}}
\end{equation*}
if $\E(W_{q}\log W_{q})\leq d\log N$ or 
\begin{equation*}
    p > \frac{2q}{q-1} +  \frac{4(d(q-\lambda)\log N + q\log \E(W^{\lambda}))}{\lambda(q-1)(\log N) \min\{ 2,\tau_{\mu}(2) \}}
\end{equation*}
if $\E(W_{q}\log W_{q})> d\log N$.

  In particular, if $\E(W_{2}\log W_{2})\leq d\log N$ and  $\tau_\mu(2) \leq 2$, then the range becomes
\[
p > \frac{2q}{q-1} +  \frac{4  \log \mathbb{E}(W^q)  }{\big(d \log N- \log \mathbb{E}(W^2) \big) (q-1) }
\]
and if   $\tau_{\mu}(2)\geq 2$, then the range becomes
\[
p > \frac{2q}{q-1} +  \frac{2  \log \mathbb{E}(W^q)  }{ (q-1) \log N}.
\]
\end{thm}

The ranges given in the above theorem can be optimised by minimising the right-hand side as a function of the free parameter $q$. In the case $q=\infty$, the above reduces to the Stein--Tomas bound and it is of course interesting to consider when the bound is optimised for a different choice of $q$.  We will see that this happens quite easily.

The restriction problem for Mandelbrot cascades in the case $d=1$ was considered explicitly in \cite[Corollary 1.6]{chenhan}.  They proved that $\|\widehat{f\mu}\|_{L^p(\rd)}\lesssim \|f\|_{L^2(\mu)}$ for all $f \in L^2(\mu)$ whenever
\[
p >  \frac{4}{\fd \mu} =\frac{4}{\tau_{\mu}(2)},
\]
noting that $\tau_\mu(2) \leq d <2$ always holds in the case $d=1$. We note that this estimate is in general not as good as the Stein—Tomas estimate and this is because the optimal value of the Frostman dimension was not used. In the subcritical case  Theorem \ref{cascade} recovers this range in the case $q=2$ and in many cases will give a better bound for some $q>2$. On the other hand, in the supercritical case (that is, whenever $\E(W_{2}\log W_{2})> d\log N$), Theorem \ref{cascade} always gives a strictly better range, even in the case $q=2$.

We consider a simple explicit example, see Figure~\ref{fig:exampleCascade}.  Let $d=N=2$ and $W$ be defined by
\begin{equation}\label{eq:mandelbrotCascade}
  W =  \left\{ \begin{array}{c|c} 2 & \text{with probability 0.1} \\
1 & \text{with probability 0.8} \\
0 & \text{with probability 0.1} 
\end{array}\right.
\end{equation}
\begin{figure}
  \begin{subfigure}{0.45\textwidth}
    \includegraphics[width=0.9\textwidth]{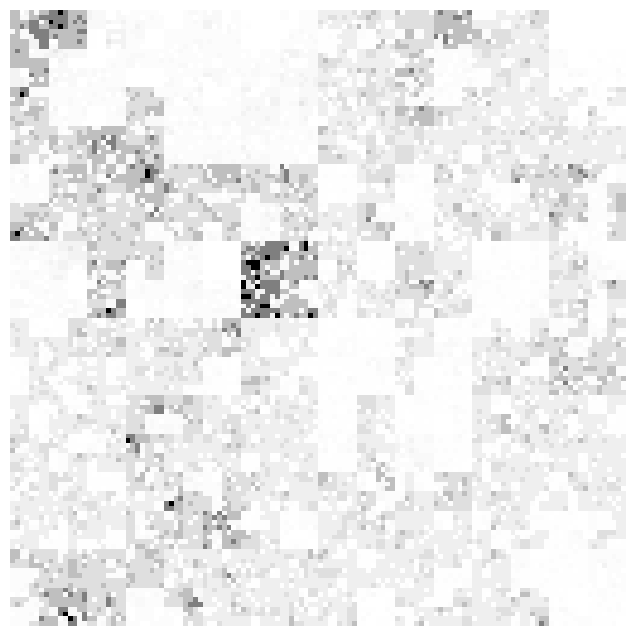}
  \end{subfigure}\hspace{5mm}%
  \begin{subfigure}{0.45\textwidth}
    \includegraphics[width=0.9\textwidth]{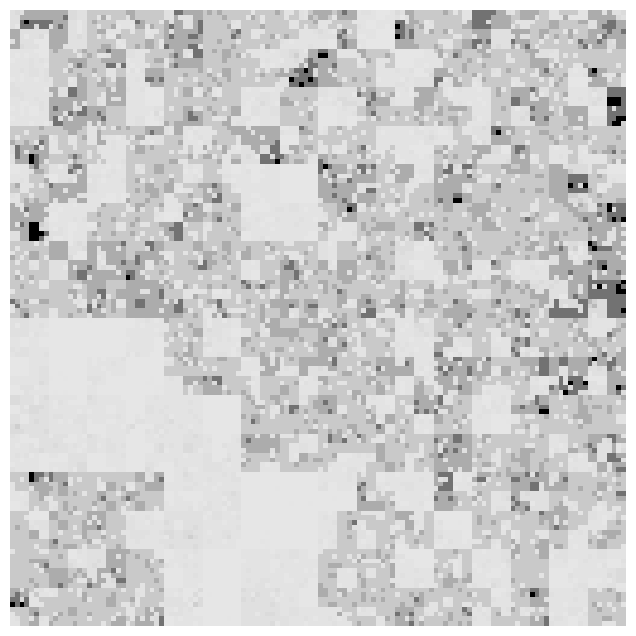}
  \end{subfigure}
  \caption{Two realisations of a Mandelbrot cascade defined with distribution \eqref{eq:mandelbrotCascade}.}\label{fig:exampleCascade}
\end{figure}
Then it is easy to check that $\mu$ is non-degenerate, sub-exponential, and  sub-critical and, moreover, $\mu$ is supported on a set of Hausdorff dimension $\approx 1.84\dots <d$.  Applying Theorem \ref{cascade}  we obtain the desired extension estimate for $p> 3.63\dots$ and   this optimal value is attained uniquely at $q \approx 4.28\dots$.  The bound at $q=2$ is $p> 4.60\dots$ and the bound at  $q=\infty$ is $p> 3.84\dots$, see Figure~\ref{fig:mandelbrotCascade}.

\begin{figure}[H]
  \includegraphics[scale=0.5]{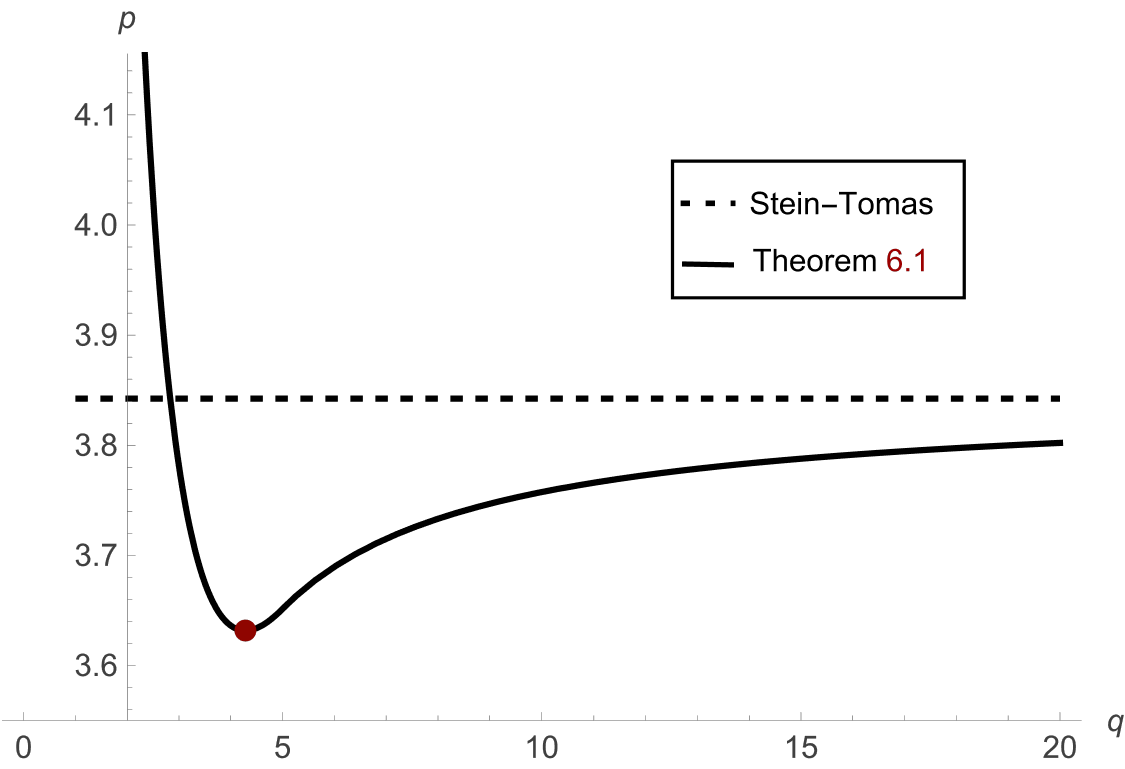}
  \caption{Plot of the lower bounds for $p$ for the extension estimate to hold for the Mandelbrot cascade with distribution \eqref{eq:mandelbrotCascade}. The dashed line is the Stein--Tomas lower bound and the solid line is the one given by Theorem~\ref{cascade} as a function of $q$. The minimum value is obtained at $q \approx 4.28\dots$.}\label{fig:mandelbrotCascade}
\end{figure}

Going further, we believe it is possible to define $W$ such that the range obtained in Theorem~\ref{cascade} is optimised for any given $q \in (1,\infty]$.  For example, let $N=2$, $c\in(0,1/N)$, $u\in(c,N^{d}]$, and $W$ be defined by
\begin{equation}\label{eq:defWcu}
  W =  \left\{ \begin{array}{c|c c} u & \text{with probability} &\frac{c}{u} \\
\frac{u(1-c)}{u-c} & \text{with probability} &\frac{u-c}{u}
\end{array}\right.
\end{equation}
for $d$ large enough to ensure $\tau_{\mu}(2)\geq 2$. Then the $q$ that optimises the bound in Theorem~\ref{cascade} is unique depending on $u$ and $c$. More precisely, for each $q>1$ there exists $c$ and large enough $u$ and $d$ such that the bound is optimal at $q$. See Figure~\ref{fig:examplecu} for an example where the bound can be optimised at $q=1.34$.
\begin{figure}[H]
  \centering
  \begin{subfigure}{0.45\textwidth}
    \includegraphics[width=\textwidth]{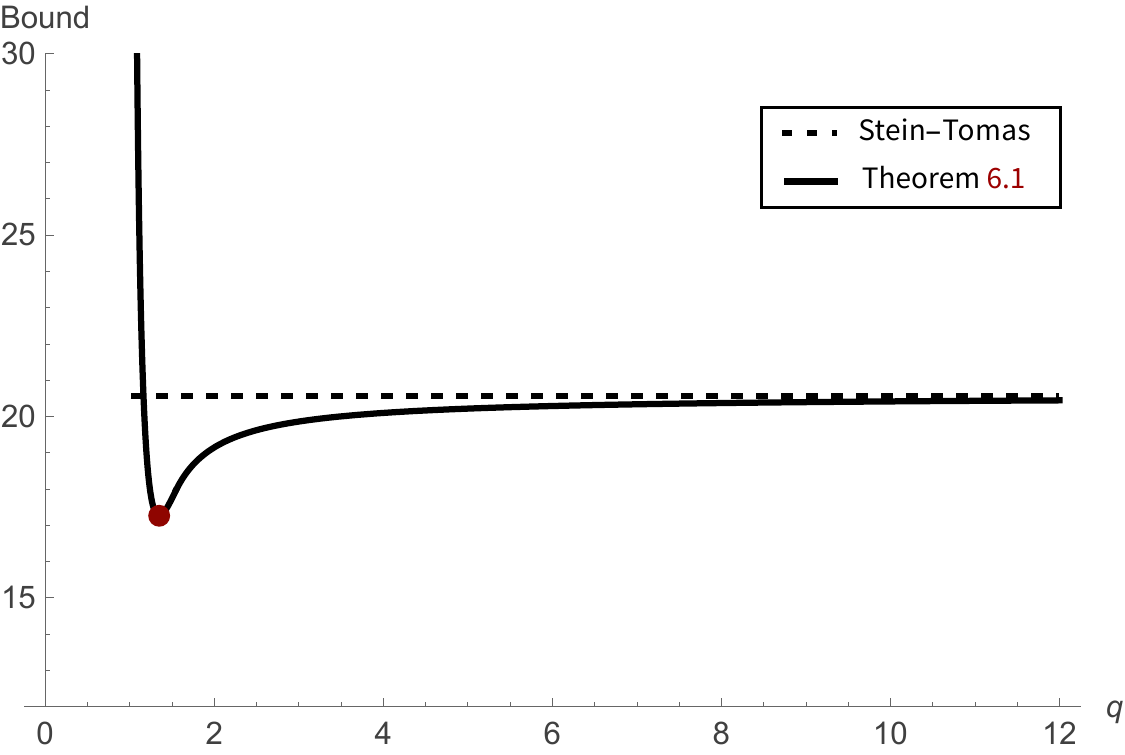}
  \end{subfigure}\hspace{8mm}%
  \begin{subfigure}{0.45\textwidth}
    \includegraphics[width=\textwidth]{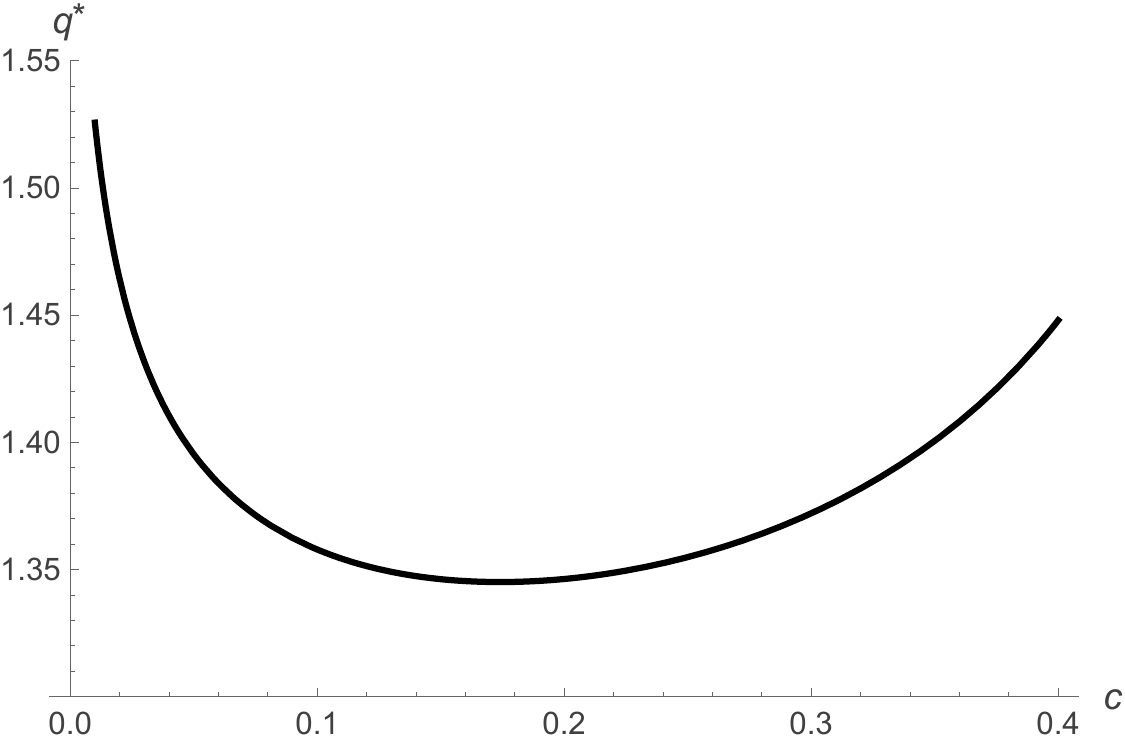}
  \end{subfigure}
  \caption{For $W$ as in \eqref{eq:defWcu}. Left: Bounds for $d = 11$, $N=2$, $c=0.22$ and $u=1000$. The minimum value is obtained at $q=1.34$. Right: Graph of the function which sends $c\in(0,0.4)$ to $q^{*}$, the value of $q$ which optimises the restriction estimate from Theorem~\ref{cascade}, for $d=11, N=2$, and $u=1000$.}\label{fig:examplecu}
\end{figure}

\end{document}